\documentclass[AMA,STIX1COL]{WileyNJD-v2}
\articletype{Report}%
\usepackage{amssymb,amsmath}
\usepackage{mathrsfs}
\usepackage{amsthm}
\usepackage{bm}
\usepackage{graphicx}
\usepackage{mwe}

\usepackage[algo2e]{algorithm2e}
\usepackage{algorithmicx}
\usepackage{algpseudocode}
\usepackage{graphics,graphicx,subfigure,float}
\usepackage{booktabs, multirow}
\usepackage{xcolor}

\newtheorem{example}{Problem}[section]
\usepackage{algpseudocode}
\usepackage{cleveref}
\usepackage{lineno}
\modulolinenumbers[1]

\graphicspath{{figures/}}

\bibstyle{NJDnatbib}

\begin{document}


\title{A short report on the preconditioned Anderson acceleration}


\author[1,3]{Kewang Chen* }

\author[2,3]{Ye Ji}

\author[3]{Matthias M\"oller}

\author[3]{Cornelis Vuik}

\authormark{K. Chen Y. Ji, M. M\"oller and C. Vuik}

\address[1]{\orgdiv{College of Mathematics and Statistics}, \orgname{Nanjing University of Information Science and Technology}, \orgaddress{\state{Nanjing, 210044}, \country{China}}}

\address[2]{\orgdiv{School of Mathematical Sciences}, \orgname{Dalian University of Technology}, \orgaddress{\state{Dalian, 116024}, \country{China}}}

\address[3]{\orgdiv{Delft Institute of Applied Mathematics}, \orgname{Delft University of Technology}, \orgaddress{\state{Delft, 2628CD}, \country{the Netherlands}}}




\corres{*Corresponding author.\\ \email{kwchen@nuist.edu.cn}}

\abstract[Summary]{
In this report, we present a versatile and efficient preconditioned Anderson acceleration (PAA) method for fixed-point iterations. The proposed framework offers flexibility in balancing convergence rates (linear, super-linear, or quadratic) and computational costs related to the Jacobian matrix. Our approach recovers various fixed-point iteration techniques, including Picard, Newton, and quasi-Newton iterations. The PAA method can be interpreted as employing Anderson acceleration (AA) as its own preconditioner or as an accelerator for quasi-Newton methods when their convergence is insufficient. Adaptable to a wide range of problems with differing degrees of nonlinearity and complexity, the method achieves improved convergence rates and robustness by incorporating suitable preconditioners. We test multiple preconditioning strategies on various problems and investigate a delayed update strategy for preconditioners to further reduce the computational costs.} 



\maketitle
\section{Introduction}
\label{sec:1}
In 1962, Anderson \cite{anderson1965iterative, anderson2019comments} developed an extrapolation algorithm to accelerate the convergence of nonlinear fixed-point iterations. Since then, this has achieved remarkable success and widespread use in various fields, such as computational chemistry (known as Pulay mixing) and electronic structure computations (referred to as Anderson mixing). Within the applied mathematics community, this technique is now recognized as Anderson acceleration (AA). Unlike Picard iterations, which utilize only one previous iterate, the Anderson Acceleration method $AA(m)$ linearly combines a set of $m$ previous iterates to approximate the minimization of the linearized fixed-point residual. Anderson acceleration methods are regarded as ``essentially equivalent" to nonlinear GMRES methods \cite{carlson1998design, oosterlee2000krylov, walker2011anderson} and the direct inversion in the iterative subspace method \cite{pulay1980convergence, lin2013elliptic}. They also share similarities with methods based on quasi-Newton updating \cite{fang2009two, haelterman2010similarities}. Recently, some authors also investigated the non-stationary AA with dynamic window sizes and damping factors. For example, Evans et al. \cite{evans2020proof} proposed a heuristic strategy for selecting damping factors based on the gain at each iteration, while Pollock and Rebholz \cite{pollock2021anderson} developed a strategy to dynamically alternate window sizes. More recently, inspired by the hybrid linear solver GMRES Recursive (GMRESR) \cite{vuik1993solution, van1994gmresr}, Chen and Vuik \cite{chen2022composite} introduced a composite Anderson acceleration method incorporating two window sizes.

AA has been widely observed to accelerate convergence, convergence analysis, however, has only been reported more recently. In 2015, Toth and Kelley \cite{toth2015convergence} initially demonstrated that the stationary version of AA without damping is locally $r$-linearly convergent if the fixed-point map is a contraction. Subsequently, Evans et al. \cite{evans2020proof} extended this result to AA with damping factors. Pollock et al. \cite{pollock2019anderson} recently applied AA to the Picard iteration for solving steady incompressible Navier–Stokes equations and showed that acceleration improved the convergence rate of the Picard iteration. Sterck and He \cite{sterck2021asymptotic} further extended this result to a more general fixed-point iteration of the form $\mathbf{x}=\mathbf{g}(\mathbf{x})$, given knowledge of the spectrum of $\mathbf{g}'(\mathbf{x})$ at the fixed-point $\mathbf{x}^*$. Additionally, Wang et al. \cite{wang2021asymptotic} investigated the asymptotic linear convergence speed of sAA when applied to the Alternating Direction Method of Multipliers (ADMM) method. Global convergence and sharper local convergence results of AA remain active research topics. For more related results about AA and its applications, we refer interested readers to the papers \cite{sterck2012nonlinear, brune2015composing, toth2017local, peng2018anderson, zhang2019accelerating, zhang2020globally, bian2021anderson, chen2022non,khatiwala2023fast} and the references therein.


Despite the extensive utilization and investigation of AA over several decades, most authors have primarily concentrated on analyzing its classical version applied to Picard iterations. Nonetheless, the potential of preconditioned Anderson acceleration (PAA) and its applicability to other iterative methods have remained largely unexplored, with only a few reported results \cite{pollock2020benchmarking, herbst2020black, li2020fast, ji2023on}. For instance, Herbst and Levitt \cite{herbst2020black} proposed using effective preconditioners $\mathbf{P}$ for solving the nonlinear equations $\mathbf{f(\mathbf{x})} = \mathbf{0}$ through preconditioned fixed-point maps $\mathbf{g(x)} = \mathbf{x} - \mathbf{P} \mathbf{f(\mathbf{x})}$. In the specific case of Newton's method, where $\mathbf{P}$ approximates the inverse of $\mathbf{f'(\mathbf{x^*})}$, Pollock and Schwartz\cite{pollock2020benchmarking} introduced and investigated the Anderson accelerated Newton method, which employs Anderson acceleration to expedite Newton's iterations. More recently, Ji et al.\cite{ji2023on} applied the preconditioned Anderson acceleration framework to a PDE-based elliptic parameterization technique in isogeometric analysis. However, the relationship between preconditioned Anderson acceleration and other iterative methods remains unclear, and the convergence analysis of PAA is currently missing.

This study presents a comprehensive investigation and analysis of a versatile preconditioned Anderson acceleration (PAA) framework. Initially, we demonstrate that PAA can encompass various fixed-point iteration techniques, including Picard, Newton, and quasi-Newton iterations. The approach of PAA can be viewed as employing Anderson acceleration both as its own preconditioner and as an accelerator for quasi-Newton methods when their convergence is insufficient. Subsequently, under suitable assumptions regarding non-singular preconditioning matrices, we establish a convergence result for the proposed PAA method. The flexibility of the PAA framework allows for striking a balance between the convergence rate (linear, super-linear, or quadratic) and the computational cost associated with approximating the Jacobian matrix. This adaptability renders the method suitable for a wide array of problems characterized by different levels of nonlinearity. By incorporating an appropriate preconditioner into the Anderson acceleration method, we achieve superior convergence rates and increased robustness. Lastly, we conduct a series of practical preconditioning strategies on various problems to enhance numerical stability and convergence speed. Through extensive numerical experiments, we demonstrate the effectiveness of the proposed update strategy.

\section{Preconditioned Anderson acceleration}\label{sec:5}

 Recall that the Picard iteration:

\begin{equation}
    \mathbf{x}_{k+1} =\mathbf{g}(\mathbf{x}_k)= \mathbf{x}_k - \mathbf{f}(\mathbf{x}_k)
\end{equation}
 can only converges linearly. However, when Anderson acceleration is used to accelerates this Picard iteration, a super-linear convergence rate could be achieved. Therefore, to further improve the numerical stability and the convergence speed, Anderson acceleration can also be used to accelerate the quasi-Newton method and even the Newton's iteration method. 

 Specifically, we introduce non-singular matrices $\mathbf{M}_k$ as the preconditioners at iteration $k$. With this, we define the following preconditioned fixed-point iteration scheme, which is also known as the quasi-Newton iteration scheme:

\begin{equation}
\label{eq:4603}
    \mathbf{x}_{k+1} = \mathbf{x}_{k}-\mathbf{M}_k^{-1} \mathbf{f}(\mathbf{x}_k).
\end{equation}
Next, we utilize the Anderson acceleration (AA) algorithm to expedite the convergence of this preconditioned iteration (\ref{eq:4603}). This approach can be seen as using AA as a preconditioner for AA itself, or as a means to accelerate quasi-Newton methods in cases where the quasi-Newton approach alone proves insufficient.

Denoting the most recent iterates by $\mathbf{x}_{k-m},\cdots,\mathbf{x}_k\in\mathbb{R}^n$ and the corresponding outputs $\mathbf{w}_{k-m}=-\mathbf{M}_k^{-1}\mathbf{f}(\mathbf{x}_{k-m}),\cdots,\mathbf{w}_k=-\mathbf{M}_k^{-1}\mathbf{f}(\mathbf{x}_k)\in\mathbb{R}^n$. Then the Anderson acceleration takes the latest $m+1$ steps into account:

\begin{equation}\label{kw-paa-1}
    \mathbf{x}_{k+1}^{AA}=\bar{\mathbf{x}}_k+\beta\bar{\mathbf{w}}_k,
\end{equation}
where the damping parameter $\beta\in(0,1]$ is a real number, $\bar{\mathbf{x}}_k$ and $\bar{\mathbf{f}}_k$ denotes the weighted average of the previous iterates and residuals, respectively. Specifically,

\begin{equation}
    \bar{\mathbf{x}}_k=\mathbf{x}_k-\sum_{j=1}^{m}\gamma_j^{k}(\mathbf{x}_{k-m+j}-\mathbf{x}_{k-m+j-1})=\mathbf{x}_k-\mathbf{E}_k\Gamma_{k},
\end{equation}

\begin{equation}
    \bar{\mathbf{w}}_k=-\mathbf{M}_k^{-1}\mathbf{f}_k-\sum_{j=1}^{m}\gamma_j^{k}(\mathbf{w}_{k-m+j}-\mathbf{w}_{k-m+j-1})=\mathbf{w}_k-\mathbf{W}_k\Gamma_{k},
\end{equation}
where $\mathbf{E}_k$ and $\mathbf{W}_k$ are given by

\begin{equation}
\begin{aligned}
    \mathbf{E}_k&=\left[(\mathbf{x}_{k-m+1}-\mathbf{x}_{k-m})\ \ (\mathbf{x}_{k-m+2}-\mathbf{x}_{k-m+1})\ \ \cdots\ \ (\mathbf{x}_k-\mathbf{x}_{k-1})\right],\\
    \mathbf{W}_k&=\left[(\mathbf{w}_{k-m+1}-\mathbf{w}_{k-m})\ \ (\mathbf{w}_{k-m+2}-\mathbf{w}_{k-m+1})\ \ \cdots\ \ (\mathbf{w}_k-\mathbf{w}_{k-1})\right]. 
\end{aligned}    
\end{equation}
The scalars $\Gamma_k=[\gamma_{1},\cdots,\gamma_{m}]^{T}$ are chosen so as to minimize the $l_2$ norm of the residual

\begin{equation}\label{kw-min-gama}
\Gamma_k=\rm{argmin}_{\Gamma\in\mathcal{R}^m}\|\mathbf{w}_k-\mathbf{W}_k\Gamma\|^2_2.
\end{equation}
Then, we obtain

\begin{equation}
\begin{aligned}
    \mathbf{x}_{k+1}^{AA}&=\bar{\mathbf{x}}_k+\beta\bar{\mathbf{w}}_k\\
    &=\mathbf{x}_k+\beta\mathbf{M}_k^{-1}\mathbf{f}(\mathbf{x}_k)-\left(\mathbf{E}_k+\beta\mathbf{W}_k\right)\left[\mathbf{W}_k^{\mathrm{T}}\mathbf{W}_k\right]^{-1}\mathbf{W}_k^{\mathrm{T}}\mathbf{M}_k^{-1}\mathbf{f}(\mathbf{x}_k).
\end{aligned}    
\end{equation}
If we set

\begin{equation}
        \mathbf{C}_k=\frac{1}{\beta}\left[\left(\mathbf{E}_k+\beta\mathbf{W}_k\right)\left[\mathbf{W}_k^{\mathrm{T}}\mathbf{W}_k\right]^{-1}\mathbf{W}_k^{\mathrm{T}}-\beta\mathbf{I}\right],
\end{equation}
the update of the preconditioned Anderson acceleration reads

\begin{equation}
    \mathbf{x}_{k+1}^{AA}=\mathbf{x}_k-\beta\mathbf{C}_k\mathbf{M}_k^{-1}\mathbf{f}(\mathbf{x}_k).
\end{equation}

Next, we give our motivations on how to choose those $\mathbf{M}_k$ and introduce several practical preconditioners.

\subsection{Motivation on choosing suitable preconditioner}
\label{sec:501}
To determine suitable preconditioners for our proposed preconditioned Anderson acceleration method, we first examine the convergence rate of the unaccelerated iteration scheme (\ref{eq:4603}). The convergence results of iteration (\ref{eq:4603}) are summarized in the following theorem.

\begin{theorem}\label{kw-th-1} For the nonlinear system (\ref{eq:4601}), let $\mathbf{x}^*$ be a solution of $\mathbf{f}(\mathbf{x})=\mathbf{0}$ and suppose that $\mathbf{f}(\mathbf{x})$ is continuously differentiable and second-order derivative exists, then the preconditioned fixed-point iteration scheme in (\ref{eq:4603})

\begin{equation*}
    \mathbf{x}_{k+1} = \mathbf{x}_{k}-\mathbf{M}_k^{-1} \mathbf{f}(\mathbf{x}_k)
\end{equation*}
converges if

\begin{equation}
\left\|\mathbf{I} - \mathbf{M}_k^{-1}\bm{\mathcal{J}}_k \right\|<1,
\end{equation}
where $\bm{\mathcal{J}}_k$ is the $n\times n$ Jacobian matrix of $\mathbf{f}(\mathbf{x})$ at iteration $k$. Moreover, if we set the local error $\mathbf{e}_{k}=\mathbf{x}^*-\mathbf{x}_k$, then the relationship between the error in any two consecutive iterates can be shown to be

\begin{equation}
\label{kw-th-1-eq1}
    \mathbf{e}_{k+1} = \left( \mathbf{I} - \mathbf{M}_k^{-1}\bm{\mathcal{J}}_k \right) \mathbf{e}_k - \frac{1}{2} \mathbf{M}_k^{-1} \mathbf{E}_k^{\mathrm{T}}\mathbf{H}_k\mathbf{E}_k + \mathcal{O}(\Vert \mathbf{e}_k \Vert^3).
\end{equation}
Here, $\mathbf{E}_k = diag \{\mathbf{e}_k^{\mathrm{T}}, \ldots, \mathbf{e}_k^{\mathrm{T}}\}$ is a diagonal matrix composed of the transposed residual vector $\mathbf{e}_k^{\mathrm{T}}$, and 

\begin{equation*}
    \mathbf{H}_k = diag \{ \mathbf{H}^{(1)}(\mathbf{x}), \mathbf{H}^{(2)}(\mathbf{x}), \ldots, \mathbf{H}^{(n)}(\mathbf{x}) \}
\end{equation*}
is a block-diagonal matrix composed of $\mathbf{H}^{(i)}(\mathbf{x})$ which is the Hessian matrix of the $i$-th nonlinear equation $\mathbf{f}^{(i)}(x)$ ($i=1,2,\ldots,n$) at iteration $k$.
\end{theorem}
\begin{proof}
Let us assume that $\mathbf{x}^*$ is the solution to the nonlinear system $\mathbf{f}(\mathbf{x}) = \mathbf{0}$, and $\mathbf{x}_k$ is an estimate for $\mathbf{x}^*$ at the $k$-th iteration, such that $\Vert \mathbf{e}_k \Vert = \Vert \mathbf{x}^* - \mathbf{x}_k \Vert \ll 1$. Based on Taylor's expansion theorem, we have:

\begin{equation}
\label{eq:5101}
\mathbf{0} = \mathbf{f}(\mathbf{x}^*) = \mathbf{f} (\mathbf{x}_k) + \bm{\mathcal{J}}_k \mathbf{e}_k + \frac{1}{2} \mathbf{E}_k^{\mathrm{T}} \mathbf{H}_k \mathbf{E}_k + \mathcal{O}(\Vert \mathbf{e}_k \Vert^3),
\end{equation}
Utilizing the preconditioned fixed-point iteration (\ref{eq:4603}), we obtain:

\begin{equation}
    \label{eq:5102}
    \mathbf{f}({\mathbf{x}_k}) = \mathbf{M}_k \left(\mathbf{x}_{k} - \mathbf{x}_{k+1}\right) = \mathbf{M}_k \left( (\mathbf{x}^* - \mathbf{x}_{k+1}) - (\mathbf{x}^* - \mathbf{x}_{k}) \right) = \mathbf{M}_k \left(\mathbf{e}_{k+1}-\mathbf{e}_{k} \right).
\end{equation}
Substituting the above formula (\ref{eq:5102}) into (\ref{eq:5101}) results in

\begin{equation}
\label{eq:5103}
\begin{aligned}
    \mathbf{0} & = \mathbf{f}(\mathbf{x}_k) + \bm{\mathcal{J}}_k \mathbf{e}_k + \frac{1}{2}\mathbf{E}_k^{\mathrm{T}}\mathbf{H}_k\mathbf{E}_k + \mathcal{O}(\Vert \mathbf{e}_k \Vert^3)\\
    &=\mathbf{M}_k \left(\mathbf{e}_{k+1} - \mathbf{e}_{k}\right) + \bm{\mathcal{J}}_k \mathbf{e}_k + \frac{1}{2}\mathbf{E}_k^{\mathrm{T}}\mathbf{H}_k\mathbf{E}_k + \mathcal{O}(\Vert \mathbf{e}_k \Vert^3).
\end{aligned}
\end{equation}
Multiplying $\mathbf{M}_k^{-1}$ on both sides of equation \eqref{eq:5103}, we have

\begin{equation}
\label{eq:5104}
\begin{aligned}
    \mathbf{0} &= \left(\mathbf{e}_{k+1} -\mathbf{e}_{k}\right) + \mathbf{M}_k^{-1} \bm{\mathcal{J}}_k \mathbf{e}_k+\frac{1}{2} \mathbf{M}_k^{-1} \mathbf{E}_k^{\mathrm{T}}\mathbf{H}_k\mathbf{E}_k + \mathcal{O}(\Vert \mathbf{e}_k \Vert^3)\\
    &= \mathbf{e}_{k+1} - \left( \mathbf{I}-\mathbf{M}_k^{-1} \bm{\mathcal{J}}_k \right)\mathbf{e}_k + \frac{1}{2} \mathbf{M}_k^{-1} \mathbf{E}_k^{\mathrm{T}} \mathbf{H}_k \mathbf{E}_k + \mathcal{O}(\Vert \mathbf{e}_k \Vert^3),
\end{aligned}
\end{equation}
where $\mathbf{I}$ is the $n \times n$ identity matrix.

Finally, we arrive at

\begin{equation*}
    \mathbf{e}_{k+1} = \left( \mathbf{I} - \mathbf{M}_k^{-1}\bm{\mathcal{J}}_k \right) \mathbf{e}_k - \frac{1}{2} \mathbf{M}_k^{-1} \mathbf{E}_k^{\mathrm{T}}\mathbf{H}_k\mathbf{E}_k + \mathcal{O}(\Vert \mathbf{e}_k \Vert^3).
\end{equation*}
In particular, convergence of the preconditioned fixed-point iteration requires

\begin{equation*}
\left\|\mathbf{I} - \mathbf{M}_k^{-1}\bm{\bm{\mathcal{J}}}_k \right\|<1.
\end{equation*}
\end{proof}

\begin{remark} For the linear problem, 

\begin{equation}
\mathbf{f}(\mathbf{x})=\mathbf{Ax-b}=\mathbf{0}, \  \mathbf{x}\in \mathbb{R}^n,
\end{equation}
suppose the matrix $\mathbf{A}$ can be split as 

\begin{equation}
    \mathbf{A=D+R},
\end{equation}
where all the off-diagonal components of the matrix $\mathbf{D}$ and the diagonal components of $\mathbf{R}$ are identically zero. By using this decomposition, the preconditioned fixed-point iteration (\ref{eq:4603}) converges if

\begin{equation}
\left\|\mathbf{I} - \mathbf{D}^{-1}\mathbf{A} \right\|<1,
\end{equation}
which recovers the results from Pratapa and Suryanarayana\cite{pratapa2016anderson}. 
\end{remark}

\begin{remark}

If we set the preconditioner 

\begin{equation}
    \label{eq:5106}
    \mathbf{M}_k = \bm{\mathcal{J}}_k,
\end{equation}
then we obtain the Newton iteration. In this case, (\ref{kw-th-1-eq1}) boils down to 

\begin{equation}
\label{eq:5107}
\begin{aligned}
      \mathbf{e}_{k+1} &= \left( \mathbf{I} - \mathbf{M}_k^{-1}\bm{\mathcal{J}}_k \right) \mathbf{e}_k - \frac{1}{2} \mathbf{M}_k^{-1} \mathbf{E}_k^{\mathrm{T}}\mathbf{H}_k\mathbf{E}_k + \mathcal{O}(\Vert \mathbf{e}_k \Vert^3)\\
       &= \left( \mathbf{I} - \bm{\mathcal{J}}_k^{-1}\bm{\mathcal{J}}_k \right) \mathbf{e}_k - \frac{1}{2} \mathbf{M}_k^{-1} \mathbf{E}_k^{\mathrm{T}}\mathbf{H}_k\mathbf{E}_k + \mathcal{O}(\Vert \mathbf{e}_k \Vert^3)\\
      &= -\frac{1}{2} \bm{\mathcal{J}}_k^{-1} \mathbf{E}_k^{\mathrm{T}}\mathbf{H}_k\mathbf{E}_k + \mathcal{O}(\Vert \mathbf{e}_k \Vert^3)\\
\end{aligned}
\end{equation}
therefore, we have

\begin{equation}
      \|\mathbf{e}_{k+1}\|\propto \Vert \mathbf{e}_k \Vert^2
\end{equation}
which means the Newton's method achieves quadratic convergence as we expected.

This implies that, when commencing with a suitably accurate initial guess and updating the preconditioner $\mathbf{M}_k$ at each iteration, our preconditioned AA scheme (\ref{eq:4603}) can attain a quadratic convergence rate, aligning with the standard Newton iteration. Nevertheless, we encounter a similar problem as faced by Newton-type solvers: the computationally intensive task of frequently updating the full Jacobian, which contradicts our initial intention of employing AA to reduce computational overhead. To address this concern, we propose several more practical preconditioning strategies aimed at effectively solving the involved nonlinear systems.
\end{remark}

Building upon the above discussions, we summarize several typical preconditioning strategies for $\mathbf{M}_k$ in Table \ref{table_2}.

\begin{table}[ht]
\caption{The selection of preconditioners $\mathbf{M}_k$ and its pros and cons} 
\centering 
\begin{tabular}{ |p{6.5cm}|p{9.5cm}|}
\hline\hline 
Preconditioners $\mathbf{M}_k$&Pros and Cons \\ [0.1ex] 
\hline 
\textbf{1. Constant preconditioner}, i.e., $\mathbf{M}_k=\alpha\mathbf{I}$, where $\alpha$ is a constant value. &This preconditioning strategy is straightforward and computationally economical. However, it requires manual tuning of the $\alpha$ value. If $\alpha$ is set to 1, it reverts to the case of no preconditioning. It's worth noting that as per \eqref{kw-th-1-eq1}, if the preconditioners $\mathbf{M}_k$ significantly deviate from the Jacobian, convergence may become limited to linear or even diverge in certain cases.\\
\hline 
\textbf{2. Full Jacobian preconditioner}, i.e.,  $\mathbf{M}_k = \bm{\mathcal{J}}_k$.&Essentially, this preconditioning strategy uses AA to accelerate Newton's method \cite{pollock2020benchmarking}. However, it comes with computational expenses since we need to update this preconditioner at each iteration.\\
\hline 
\textbf{3. Diagonal Jacobian or Block-diagonal Jacobian preconditioner}, i.e., $\mathbf{M}_k = \rm{diag}(\bm{\mathcal{J}}_k)\ or\ \mathbf{M}_k = \rm{diagBlock}(\bm{\mathcal{J}}_k)$.&To strike a balance between a ``good preconditioner" and ``computational overhead", the (block) diagonal Jacobian preconditioner can be a promising candidate due to its relatively inexpensive computation.\\
\hline 
\textbf{4. Other approximations of Jacobian as preconditioners}. &For example, we can use rank one correction and rank two correction of the inverse of Jacobian matrix in L-BFGS \cite{liu1989limited} as a preconditioner. Besides, we could also use AA itself as a preconditioner for AA\cite{chen2022composite}. These preconditioners are computationally cheaper and they can approximate the inverse of Jacobian matrix directly.\\
\hline 
\textbf{5. Updating the preconditioners $\mathbf{M}$ every $N_{update}$ steps instead of updating them each iteration}. &Unlike Newton-type solvers and gradient-based methods, the AA iteration scheme involves a linear combination of several previous iterates. This characteristic allows us to avoid frequent updates of preconditioners, which can be computationally expensive.
\\[0.5ex] 
\hline 
\end{tabular}
\label{table_2}
\end{table}

\subsection{Overview of the proposed algorithm and its relations with other methods}
\label{sec:502}

In this section, we present the details of the proposed algorithm and explore its connections with other methods. The general preconditioned Anderson acceleration scheme, incorporating a dynamic preconditioning strategy, is summarized in the following Algorithm \ref{alg:1}.


\begin{algorithm}
    \caption{Preconditioned Anderson acceleration: \textbf{PAA(m)}}
    \label{alg:1}
    \KwIn{\quad $\mathbf{x}_0$: Initial guess; \\
    \quad \quad \quad \quad $\mathbf{f}(\mathbf{x})$: Nonlinear system; \\
    \quad \quad \quad \quad $m$: Window size used in AA; \\
    \quad \quad \quad \quad  $N_{max}$: Maximum iterations; \\
    \quad \quad \quad \quad  $tol$: Tolerance of residual norm for convergence test; \\
    \quad \quad \quad \quad  $N_{update}$: Update the preconditioner every $N_{update}$ steps.}
    \KwOut{$\mathbf{x}_{k+1}$: Numerical solution\\}
    
    \For{$k=0 \mathrm{;}$ $k \le N_{max} \mathrm{;}$ $++k$}
    {
        Compute residual vector $\mathbf{f}_{k} =\mathbf{f}(\mathbf{x}_k)$; {\color{gray} \hfill }
        
        \vspace{0.2 cm}
        {\color{gray}  \tcp{Check the termination criteria}}
        \If{ $\Vert \mathbf{f}_k \Vert < tol$}{
        break;}
    
        \vspace{0.2 cm}
        {\color{gray}  \tcp{Update preconditioner}}
        
        \If{$k$ is evenly divisible by $N_{update}$}{
             Compute preconditioning matrix $\mathbf{M_k}$; {\color{gray} \hfill \tcp{\Cref{table_2}}}
        }
        
        \vspace{0.2 cm}
        {\color{gray}  \tcp{Anderson acceleration}}
        Compute $m_{k}=\min\{m,k\}$; \\
        
        Compute preconditioned residual vector $\mathcal{F}_k$ by solving $\mathbf{M_k} \mathcal{F}_k = -\mathbf{f}_k$; \\
        
        
        Determine $\bm{\alpha}^{(k)}=\left(\alpha_0^{(k)},\cdots,\alpha_{m_k}^{(k)}\right)^{\mathrm{T}}$ that solves
        \begin{equation*}
        \begin{aligned}
             \mathop{\arg\min} \limits_{\bm{\alpha}=(\alpha_0,\cdots,\alpha_{m_k})^{\mathrm{T}}}\quad & \left \Vert \sum_{i=0}^{m_k} \mathcal{F}_{k-m_k+i} \bm{\alpha} \right \Vert ^2_2 \\
            \text{s.t.}\quad &\sum_{i=0}^{m_k}\alpha_i = 1; 
        \end{aligned}
        \end{equation*}
        

        Update $\displaystyle \mathbf{x}_{k+1}= \sum_{i=0}^{m_k}\alpha_{i}^{(k)} \left( \mathbf{x}_{k-m_k+i} + \mathcal{F}_{k-m_k+i} \right)$;
    }
\end{algorithm}

\begin{remark}The proposed preconditioned Anderson acceleration $(\mathbf{PAA}(m))$ method can be viewed as a nonlinear acceleration method applied to a general iteration scheme:

\begin{equation}
\mathbf{x}_{k+1}=\mathbf{U}(\mathbf{x}_k),
\end{equation}
where $\mathbf{U}(\mathbf{x})$ is an update function. The $\mathbf{PAA}(m)$ method can accelerate various fixed-point iteration methods, including the Picard iteration, the Newton iteration, and the quasi-Newton iteration. Specifically, the $\mathbf{PAA}(m)$ method accelerates the Picard iteration when $\mathbf{M}_k=\mathbf{I}$, i.e., $\mathbf{U}(\mathbf{x}_k)=\mathbf{x}_{k}- \mathbf{f}(\mathbf{x}_k)$ (see Anderson\cite{anderson1965iterative}), accelerates the Newton iteration when $\mathbf{M}_k = \bm{\mathcal{J}}_k$, i.e., $\mathbf{U}(\mathbf{x}_k)=\mathbf{x}_{k}-[\bm{\mathcal{J}}_k]^{-1}\mathbf{f}(\mathbf{x}_k)$ (see Pollock and Schwartz \cite{pollock2020benchmarking}), and accelerates the quasi-Newton iteration when the preconditioner matrix $\mathbf{M}_k$ lies between these two matrices and $\mathbf{U}(\mathbf{x}_k)=\mathbf{x}_{k}- \mathbf{M}_k^{-1}\mathbf{f}(\mathbf{x}_k)$. Moreover, the $\mathbf{PAA}(m)$ method can recover all the above unaccelerated iteration methods by setting the window size $m=0$ and using different $\mathbf{M}_k$ in $\mathbf{PAA}(m)$.

The key advantage of the $\mathbf{PAA}(m)$ method lies in its ability to easily control and manipulate the trade-off between convergence rate (linear, superlinear, or quadratic) and the cost of computing or approximating the Jacobian matrix (identity matrix, diagonal Jacobian matrix, full Jacobian matrix, or other approximation Jacobian matrix). This control can be achieved by tuning the window size $m$, the preconditioners $\mathbf{M}_k$, and the frequency of updating these preconditioners $\mathbf{N}_{update}$ in $\mathbf{PAA}(m)$.
\end{remark}

\begin{remark}
    Other aspects to further improve the performance of AA: Using damping factor or dynamical damping factors $\beta$, see Evans et al.\cite{evans2020proof} and Chen and Vuik\cite{chen2022non}; Using filtering for Anderson acceleration, see Haelterman et al.\cite{haelterman2016improving} and Pollock and Rebholz\cite{pollock2022filtering}; Using multiple window sizes, see Pollock and Rebholz\cite{pollock2021anderson}.
\end{remark}

\section{Numerical experiments}
\label{sec:6}


\subsection{Benchmark results}
The problems considered here are the systems of nonlinear equations from the standard test set in literature. Throughout this section, each method was terminated on the residual reduced below the tolerance $\|\mathbf{f}(\mathbf{x_k})\|<10^{-10}$. All these experiments are done in the MATLAB 2021b environment.

\begin{example}\label{p1} \rm{[\textbf{2D nonlinear polynomial systems\cite{kelley1983new}}]} Apply preconditioned AA to solve the following $\mathbf{f}(\mathbf{x})=\mathbf{0}$ with

\begin{equation}\label{kw_example_1}
\left \{
    \begin{aligned}
        f_1(x_1,x_2)&=(x_1-1)+(x_2-3)^2,\\
        f_2(x_1,x_2)&=\epsilon(x_2-3)+\frac{3}{2}(x_1-1)(x_2-3)+(x_2-3)^2+(x_2-3)^3.
    \end{aligned}
\right.
\end{equation}
\end{example}
In this simple example, we know from Kelley and Suresh\cite{kelley1983new} that there are two distinct roots  $\mathbf{x}_+=(1,3)^{\mathrm{T}}$ and $\mathbf{x}_-=(1-\eta^2,3+\eta)^{\mathrm{T}}$ with $\eta=1-\sqrt{1+2\epsilon}$ for $\epsilon>0$. When $\epsilon=0$, the two roots collapse into the same solution $\mathbf{x}=(1,3)^{\mathrm{T}}$. We apply preconditioned AA with different preconditioning strategies to solve this problem for both $\epsilon=0$ and $\epsilon=10^{-6}$. For each method, we use $5$ random initial guesses $(x_0,y_0)\in[-1,3]\times[1,5]$ in our numerical experiments.

From Fig.~\ref{fig-example-1}, we observe that the original Anderson acceleration method $AA(1)$, without preconditioning, exhibits poor performance and may even diverge for certain initial guesses in both cases. A simple constant preconditioning strategy, $\mathbf{M}_k = \alpha\mathbf{I}$ with $\alpha=0.1$, improves the convergence rate slightly, but it remains slow. On the other hand, when better preconditioners, such as the diagonal Jacobian and the full Jacobian preconditioning strategies, are employed, the performance is significantly improved. However, it is essential to consider the computational costs of these preconditioners when comparing their performance for solving large systems of nonlinear equations.

\begin{figure}
    \centering
    \subfigure[$\epsilon=0$]{
     \includegraphics[width=.5\linewidth]{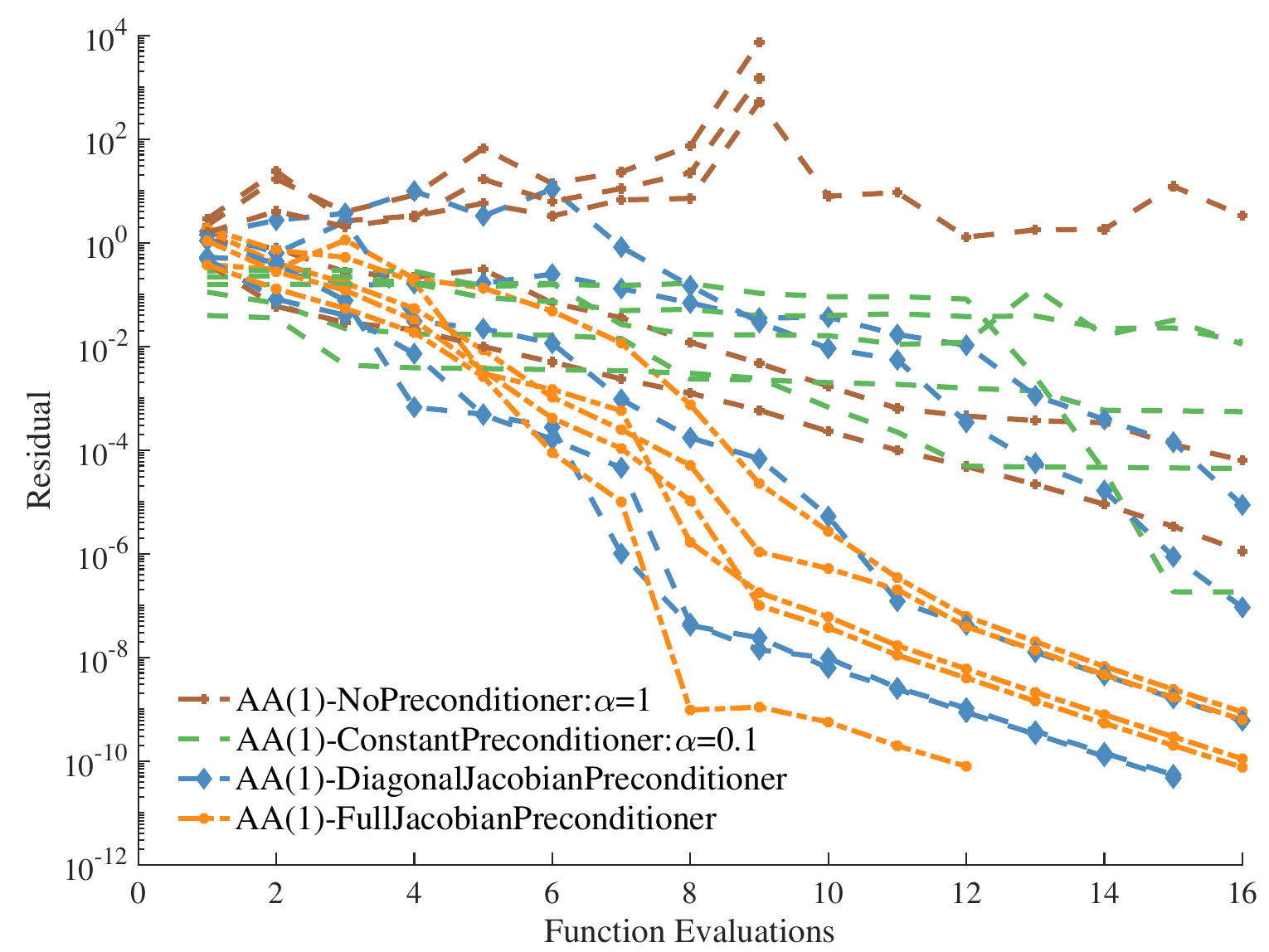}
    \label{fig:example1_5}}
    \quad
    \subfigure[$\epsilon=10^{-6}$]{
     \includegraphics[width=.5\linewidth]{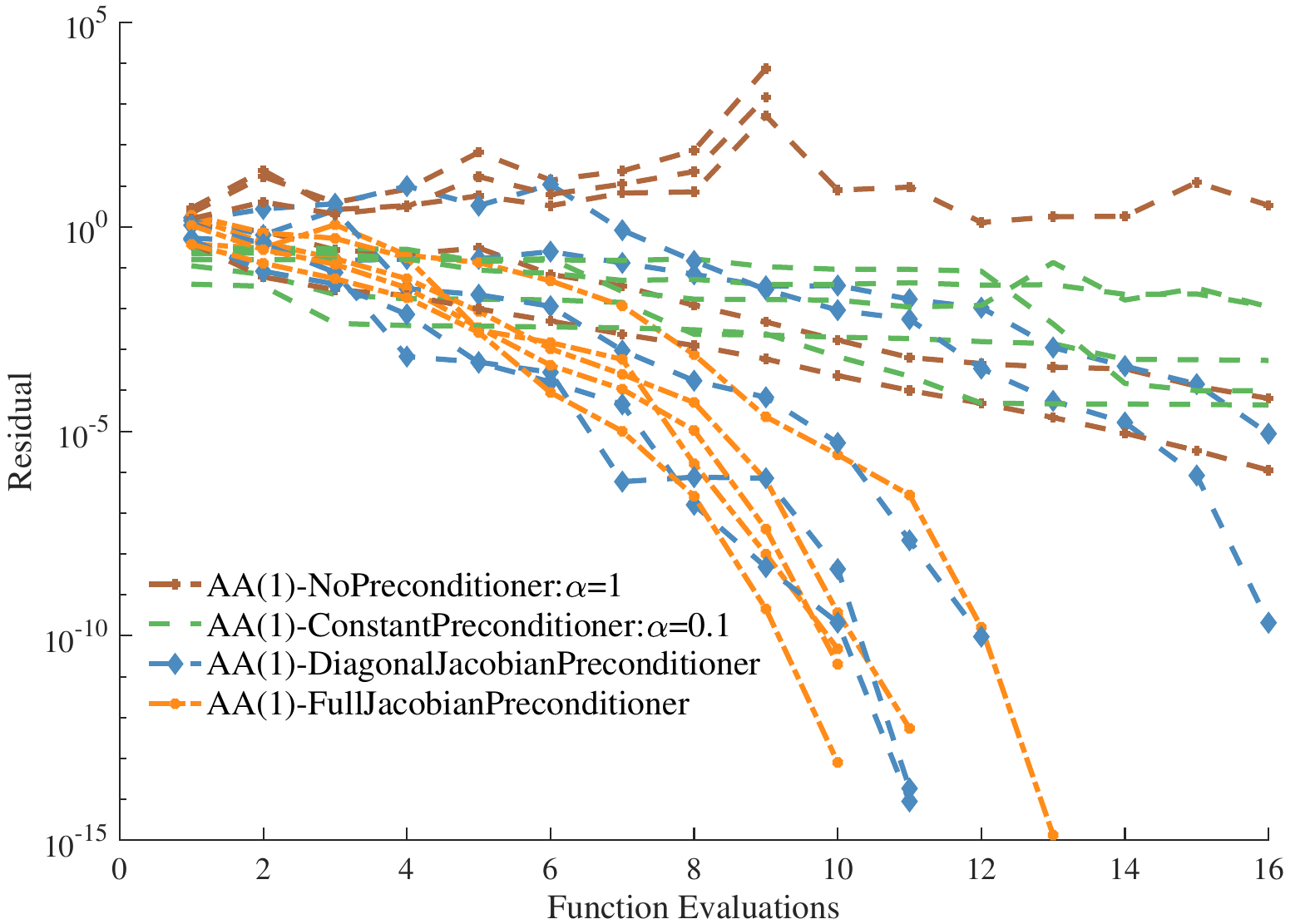}
    \label{fig:exampe1_50}}
    \caption{Using Preconditioned AA to solve a 2D nonlinear polynomial system with 5 random initial guesses.}\label{fig-example-1}
\end{figure}

\begin{example}\label{p2}  \rm{[\textbf{Trigonometric nonlinear equations, n=5,50,500}}\cite{more1981testing}] Apply preconditioned AA to solve the following larger nonlinear trigonometric systems. Let

\begin{equation}h_i(x)=n-\sum_{j=1}^{n}\cos(x_j)+i[1-\cos (x_i)]-\sin (x_i),\ i=1,2,\cdots,n,\end{equation}
we create a manufacturing fixed-point $\mathbf{x}^*=(\frac{\pi}{4},\cdots,\frac{\pi}{4})^{\mathrm{T}}$ of the following system

\begin{equation}n-\sum_{j=1}^{n}\cos(x_j)+i(1-\cos (x_i))-\sin (x_i)=h_i(\frac{\pi}{4}),\ i=1,2,\cdots,n.\end{equation}
Then, we define the nonlinear equations as:

\begin{equation}\label{kw_examle_2}f_i(x)=\left(h_i(x)-h_i(\frac{\pi}{4})\right),\ i=1,2,\cdots,n.\end{equation}
with real solution $\mathbf{x}^*=(\frac{\pi}{4},\cdots,\frac{\pi}{4})^{\mathrm{T}}$.
\end{example}

In this numerical test, we adopt an example from Mor\'e et al.\cite{more1981testing} such that a manufacturing solution is created. The initial guesses $\mathbf{x}_0\in \mathbb{R}^n$ used here are five random vectors close to $\mathbf{x}^*$ such that $\mathbf{x}_0\in(\pi/4-0.05,\pi/4+0.05)\times\cdots\times(\pi/4-0.05,\pi/4+0.05)$. We apply preconditioned AA with different preconditioning strategies to solve equations \eqref{kw_examle_2} for $n=5, 50, 500$. As the dimension $n$ increases, we use larger window sizes for Anderson acceleration. The results of our experiments are illustrated in Fig.~\ref{fig-example-2}.

When $n=5$, we see from Fig.~\ref{fig:example2_5} that the straightforward constant preconditioning approach using $\mathbf{M}_k=\alpha\mathbf{I}$ with $\alpha=0.1$ and $\alpha=0.01$, yields poorer results than the original $AA(3)$ method. However, the diagonal Jacobian preconditioning and the full Jacobian strategies performs much better than original $AA(3)$ as we expected. This indicates that the simple constant preconditioner may not always works and a good preconditioner is very important in solving nonlinear equations.

When $n=50$ and $n=500$, as we can see from Fig.~\ref{fig:exampe2_50} and Fig.~\ref{fig:example2_500} that the original Anderson acceleration without preconditioning does not converge. The simple constant preconditioning strategy performs slightly better but still diverges or converges very slowly. Again, if we choose better preconditioners, the diagonal Jacobian preconditioning and the full Jacobian strategies converge very fast. 

However, as $n$ grows larger, the cost of computing a full Jacobian preconditioner appears to be progressively increasing, see Table.~\ref{table:t1}. For example, when $n=500$, the Anderson acceleration with diagonal Jacobian preconditioner only needs around 0.09 seconds to find the solution while the full Jacobian preconditioning strategy (Anderson accelerates Newton's iteration) needs almost around 16 seconds to achieve the same accuracy. Moreover, to further improve the time efficiency of preconditioned AA, we can also update our preconditioners every $N_{update}$ steps instead calculating it at each iteration. For instance, when employing Anderson acceleration with the full Jacobian preconditioning strategy, setting the number of steps between preconditioner updates to $N_{update}=2$ leads to superior performance in comparison to computing preconditioners at every iteration, as demonstrated by the results presented in Table~\ref{table:t1}. 
 

\begin{figure}
    \centering
    \subfigure[n=5]{
    \includegraphics[width=0.45\linewidth]{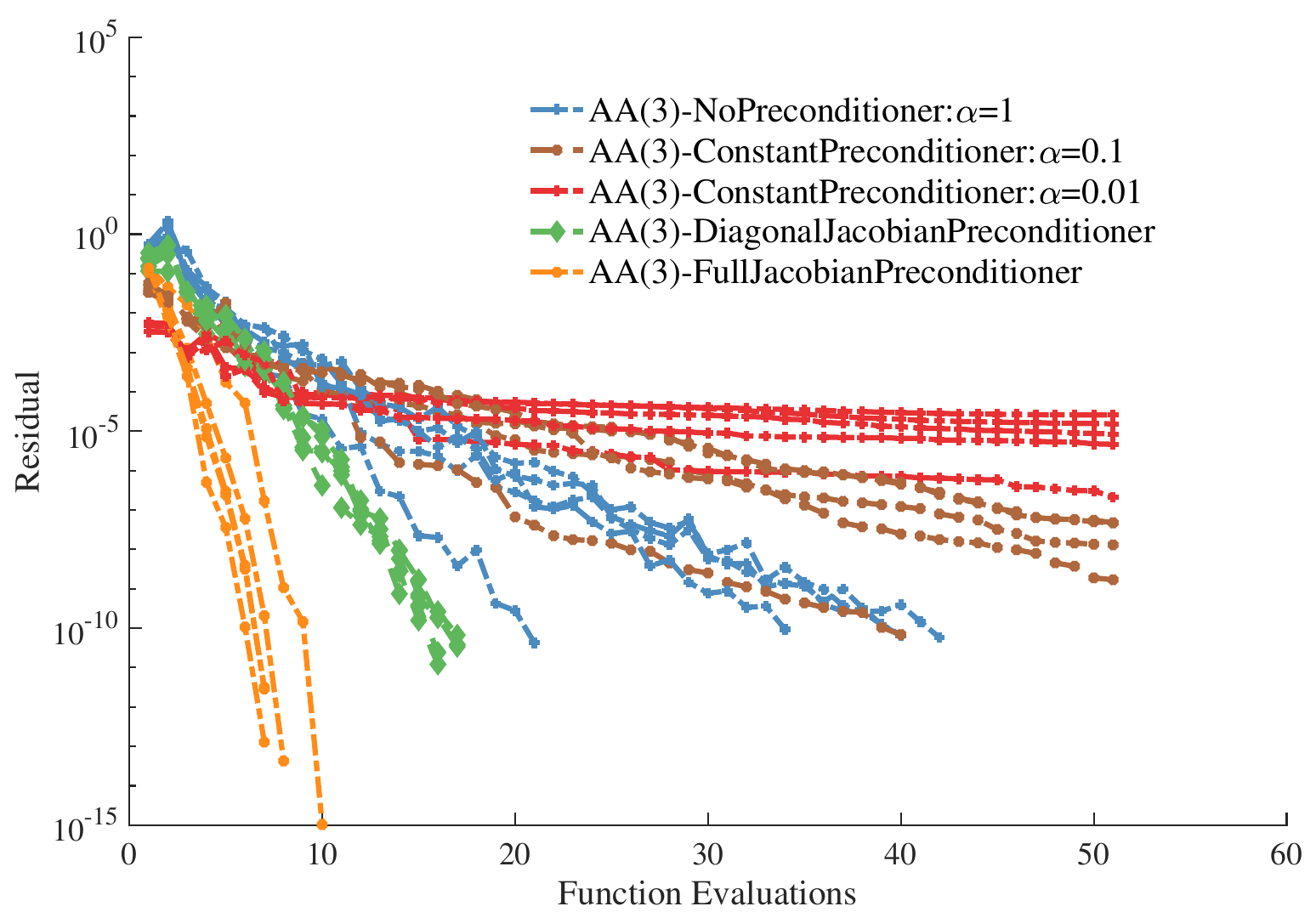}
    \label{fig:example2_5}}
    \quad
    \subfigure[n=50]{
    \includegraphics[width=0.45\linewidth]{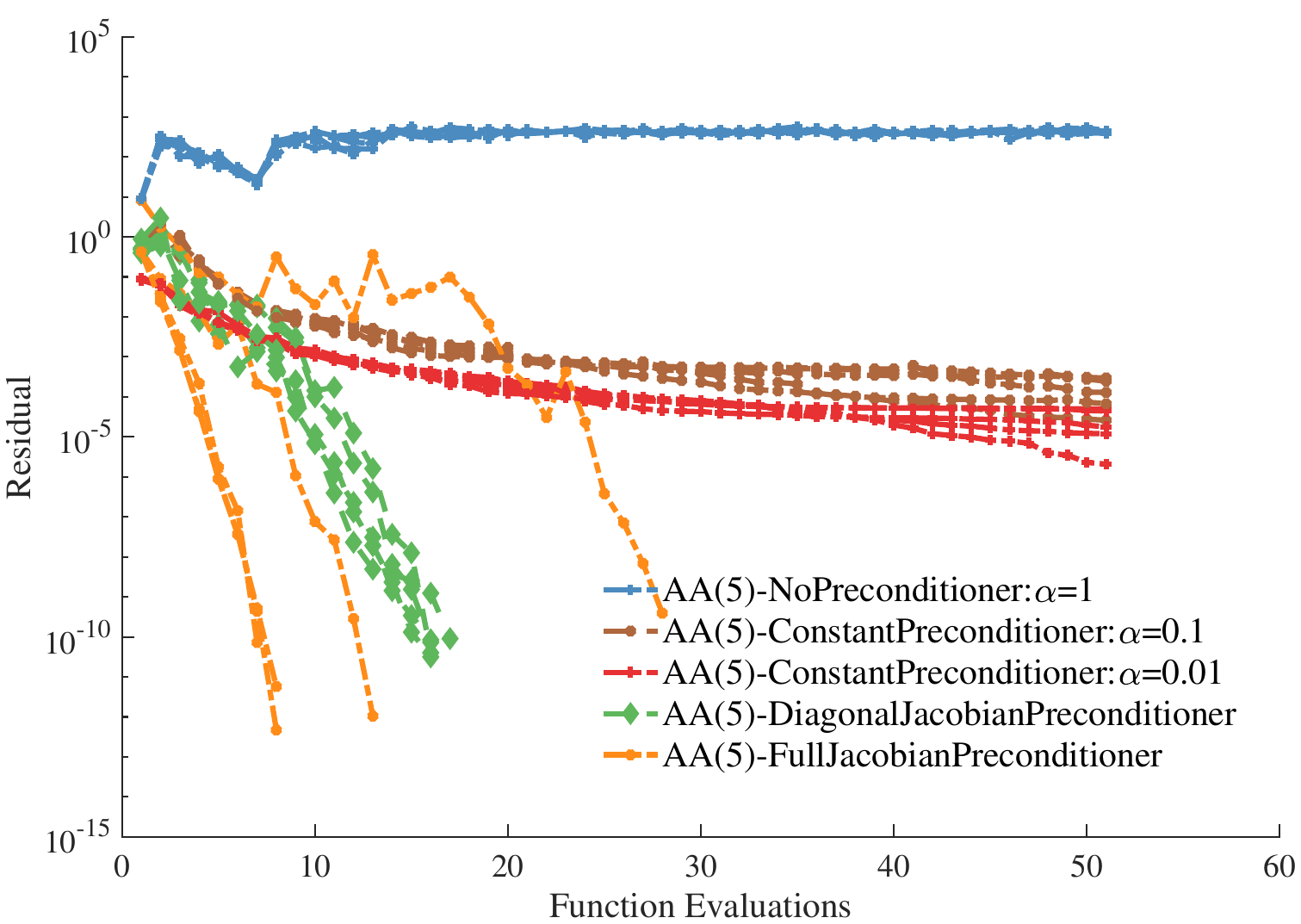}
    \label{fig:exampe2_50}}
    \quad
    \subfigure[n=500]{
    \includegraphics[width=0.45\linewidth]{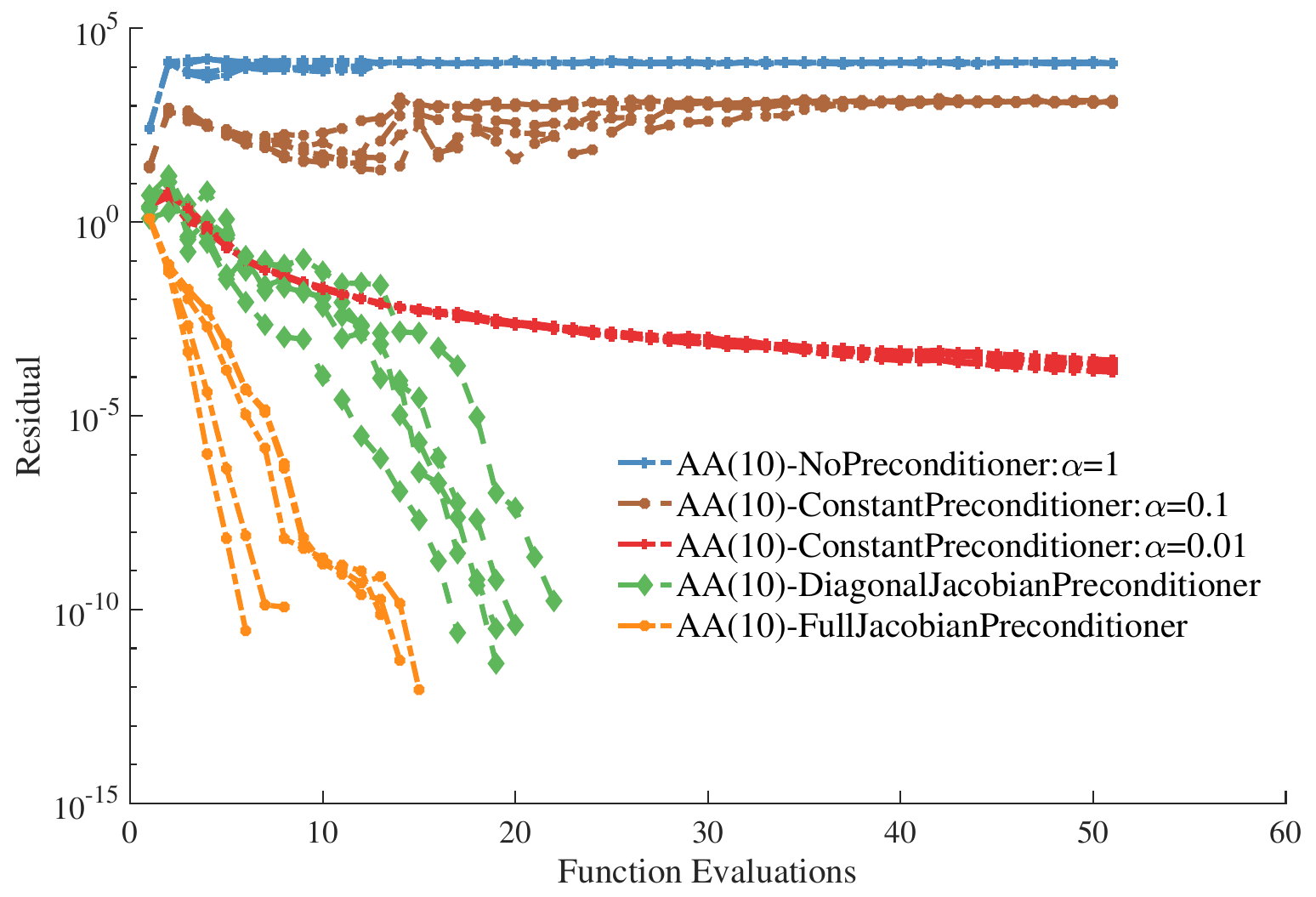}
    \label{fig:example2_500}}
    \caption{Using preconditioned AA to solve trigonometric nonlinear equations.}\label{fig-example-2}
\end{figure}

\begin{table}[ht]
\caption{Total average time (seconds) used for different preconditioned AA methods.} 
\centering 
\begin{tabular}{c c c c} 
\toprule 
Methods&\ \ $n=5$ &\ \ $n=50$ &\ \  $n=500$  \\ [0.5ex] 
\hline 
AA & 0.0040 &- & -\\
DiagJacobian preconditioner& 0.0030 & \textbf{0.0040} & \textbf{0.0909} \\
FullJacobian preconditioner& 0.0038 & 0.0411 &  15.8453 \\
FullJacobian preconditioner, $N_{update}=2$& \textbf{0.0021} & 0.0322 &  11.7176 \\[0.5ex] 
\bottomrule 
\end{tabular}
\label{table:t1} 
\end{table}

\begin{example}\rm{[\textbf{The\ nonlinear Bratu\ problem}\cite{walker2011anderson}]} The Bratu problem is a nonlinear PDE boundary value problem as follows:

\begin{equation*}
\begin{aligned}
    \Delta u +\lambda\ e^{u} &= 0,\ \ in\ \ \Omega=[0,1]\times[0,1], \\
    u &= 0,\ \ on \ \ \partial \Omega. 
\end{aligned}
\end{equation*}

\end{example}

This example is from Walker\cite{walker2011anderson},  where a stationary Anderson acceleration with window size $m=50$ is used to solve the Bratu problem. This problem has a long history, we refer the reader to Glowinski et al.\cite{glowinski1985continuation} and the references in those papers. 

In this experiment, we employe a centered-difference discretization approach to approximate the solution of the Bratu problem on two different grid sizes: $32 \times 32$ and $64 \times 64$. The value of parameter $\lambda$ in the Bratu problem was set to 6, and the initial guess for the solution was chosen as $u_0=(1,1,\cdots,1)^{T}$ for all cases. It was observed that the Picard iteration diverged when no preconditioning was applied. We also note here that the preconditioning matrix of the linear part refers to the matrix $A$, where $A$ is a matrix for the discrete Laplace operator. The numerical results are shown in Fig.~\ref{bratu}.

From Fig.~\ref{bratu}, we observe that Anderson acceleration, utilizing a diagonal Jacobian preconditioner (either for only the linear component or for the complete nonlinear system), notably expedites the Picard iteration. Although the variant utilizing the full Jacobian achieves even faster convergence, computing the complete Jacobian incurs significant computational expense. It necessitates more time to achieve comparable accuracy compared to the diagonal preconditioning strategies. Consequently, deriving a "reasonably good and affordable" approximation of the complete Jacobian is crucial when solving nonlinear systems with high dimensions. The preconditioners utilized in preconditioned Anderson acceleration (PAA) must be adjusted accordingly to obtain optimal acceleration.

\begin{figure}
    \centering
    \subfigure[32 $\times$ 32]{
     \includegraphics[width=.55\linewidth]{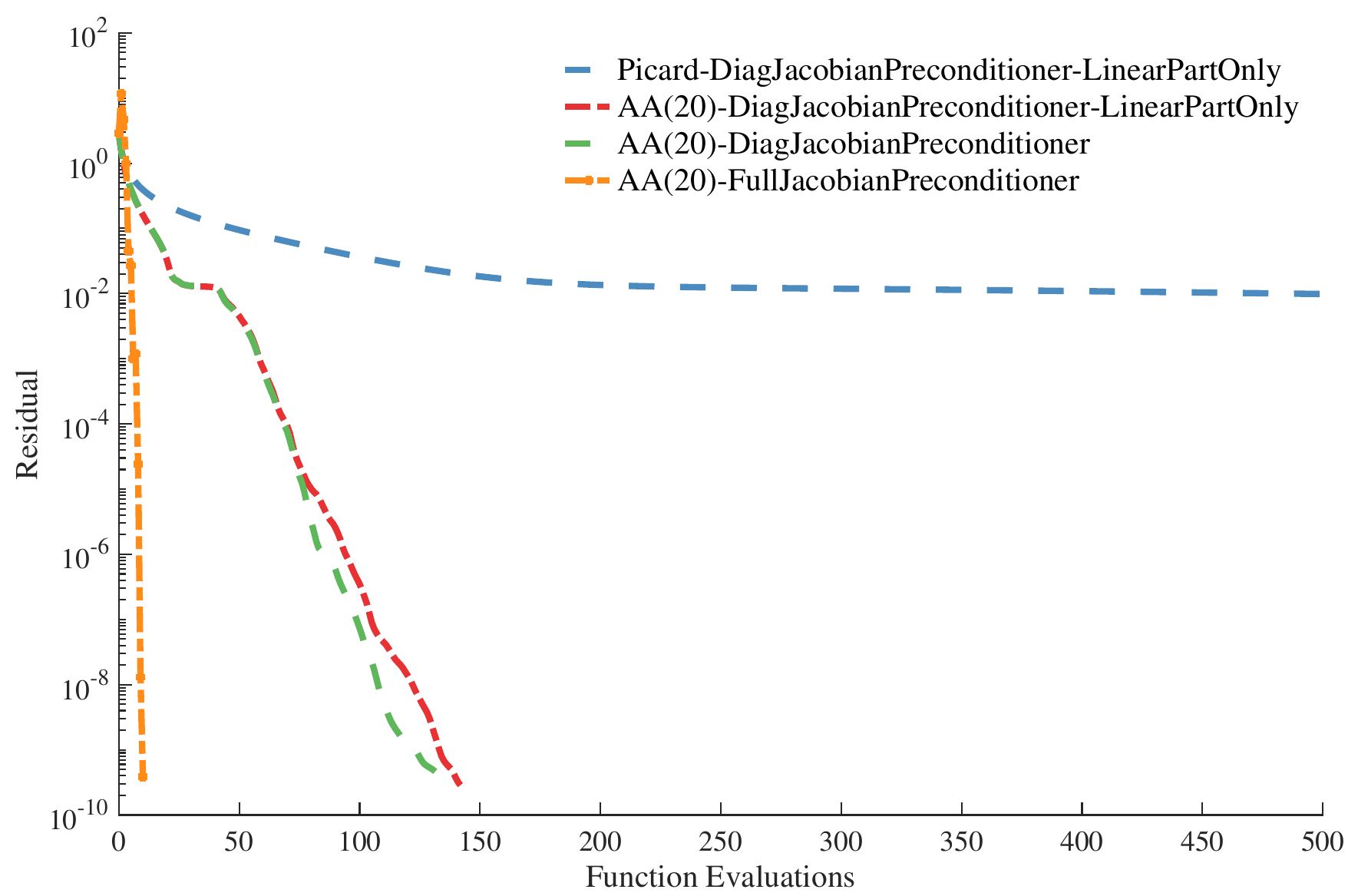}
    \label{fig:example3_5}}
    \quad
    \subfigure[64 $\times$ 64]{
     \includegraphics[width=.55\linewidth]{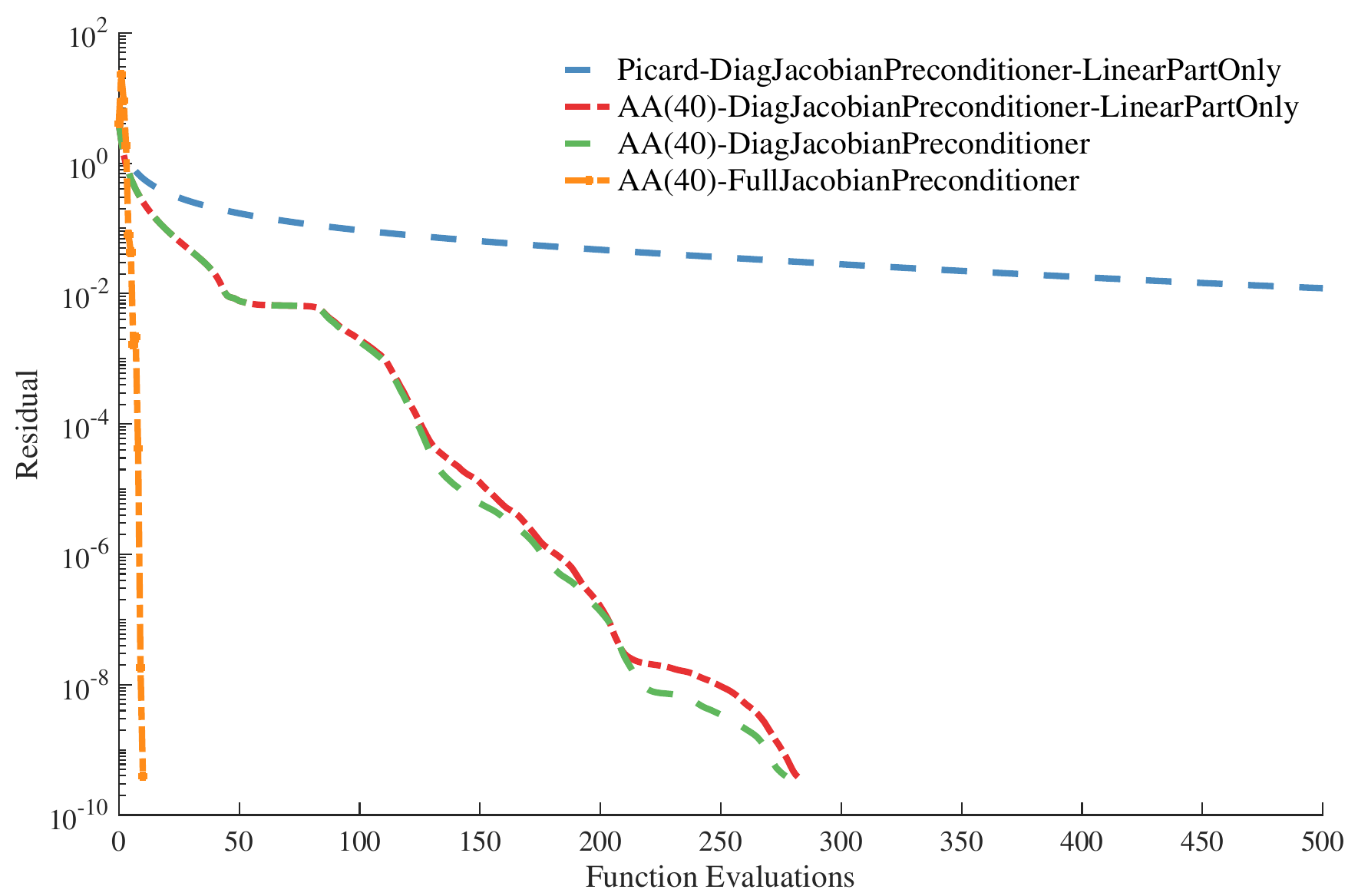}
    \label{fig:exampe3_50}}
    \caption{Using preconditioned AA to solve the nonlinear Bratu problem.}\label{bratu}
\end{figure}

\begin{example}\rm{[\textbf{The\ stationary\ nonlinear\ convection-diffusion\ problem}\cite{chen2022composite}]} Solve the following 2D nonlinear convection-diffusion equation in a square region:

\begin{equation}\epsilon(-u_{xx}-u_{yy})+(u_x+u_y)+ku^2=f(x,y),\ \  (x,y)\in \Omega =[0,1]\times[0,1] \end{equation}
with the source term 

\begin{equation}f(x,y)=2\pi^2\sin(\pi x)\sin(\pi y)\end{equation}
and zero boundary conditions: $u(x,y)=0$ on $\partial \Omega$.
\end{example}

In this numerical experiment, we use a centered-difference discretization on a $32\times 32$ grid. We choose $\epsilon=0.1$ and $\epsilon=0.01$, respectively. Those $\epsilon$ values indicate the competition between the diffusion and convection effect. We take $k=3$ in the above problem and use $u_0=(1,1,\cdots,1)^{T}$ as an initial approximate solution in all cases. Again, the preconditioning matrix of the linear part refers to the matrix $A$, where $A$ is a matrix for the discrete Laplace operator.

The results are shown in Fig.~\ref{diffusion}, which are consistent with that of solving the Bratu problem. Particularly, when $\epsilon=0.01$, the linear part becomes less important and the nonlinear term dominates. Fig.\ref{fig:exampe4_50} illustrates that solely utilizing the linear portion of the diagonal Jacobian in a simple Picard iteration leads to divergence. However, our preconditioned Anderson acceleration method, utilizing the same preconditioner, achieves convergence. Furthermore, implementing a superior preconditioner (i.e., the diagonal Jacobian of the complete nonlinear equations) results in significantly faster convergence, as anticipated.

\begin{figure}
    \centering
    \subfigure[$\epsilon=0.1$]{
     \includegraphics[width=.55\linewidth]{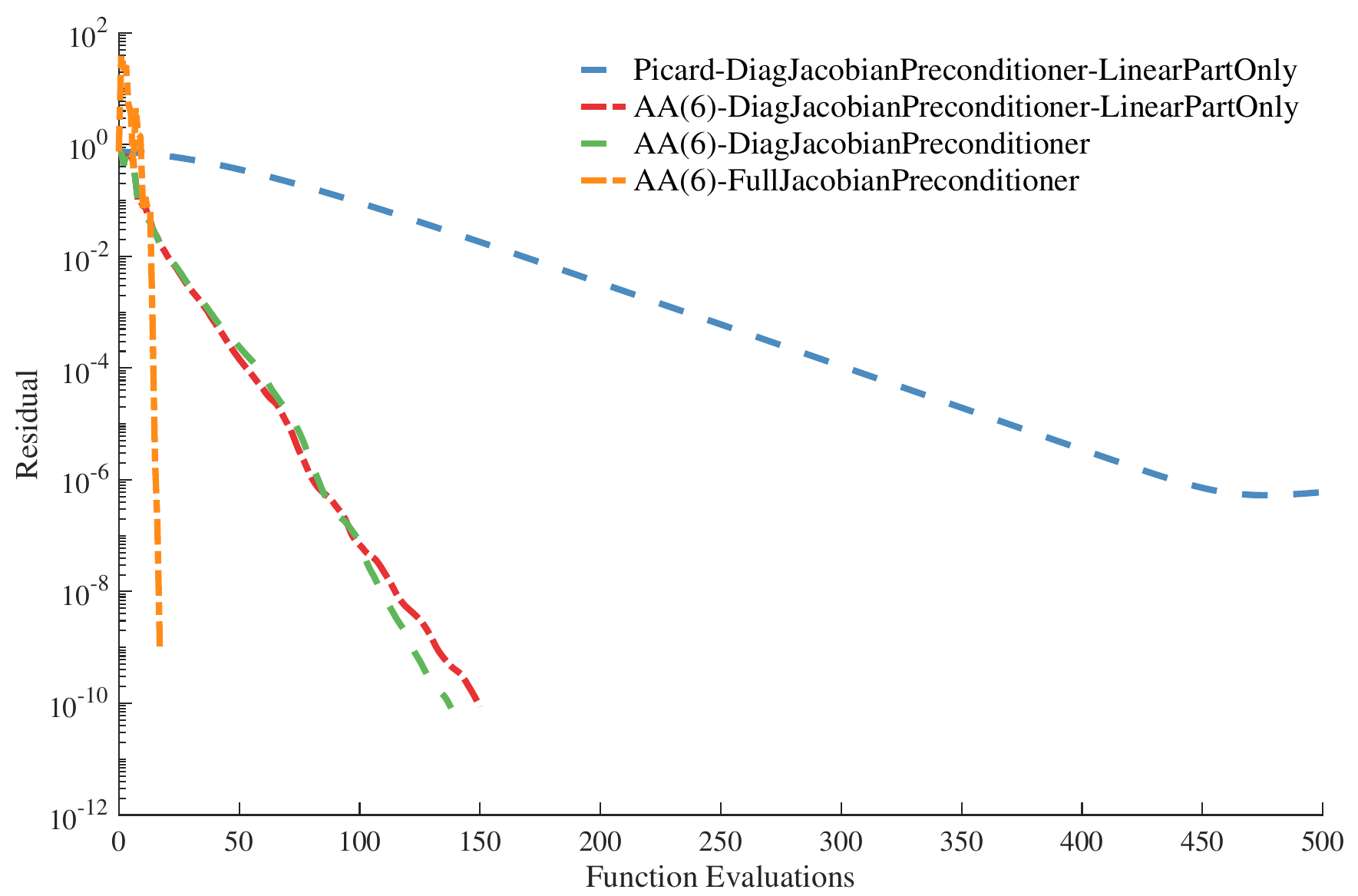}
    \label{fig:example4_5}}
    \quad
    \subfigure[$\epsilon=0.01$]{
     \includegraphics[width=.55\linewidth]{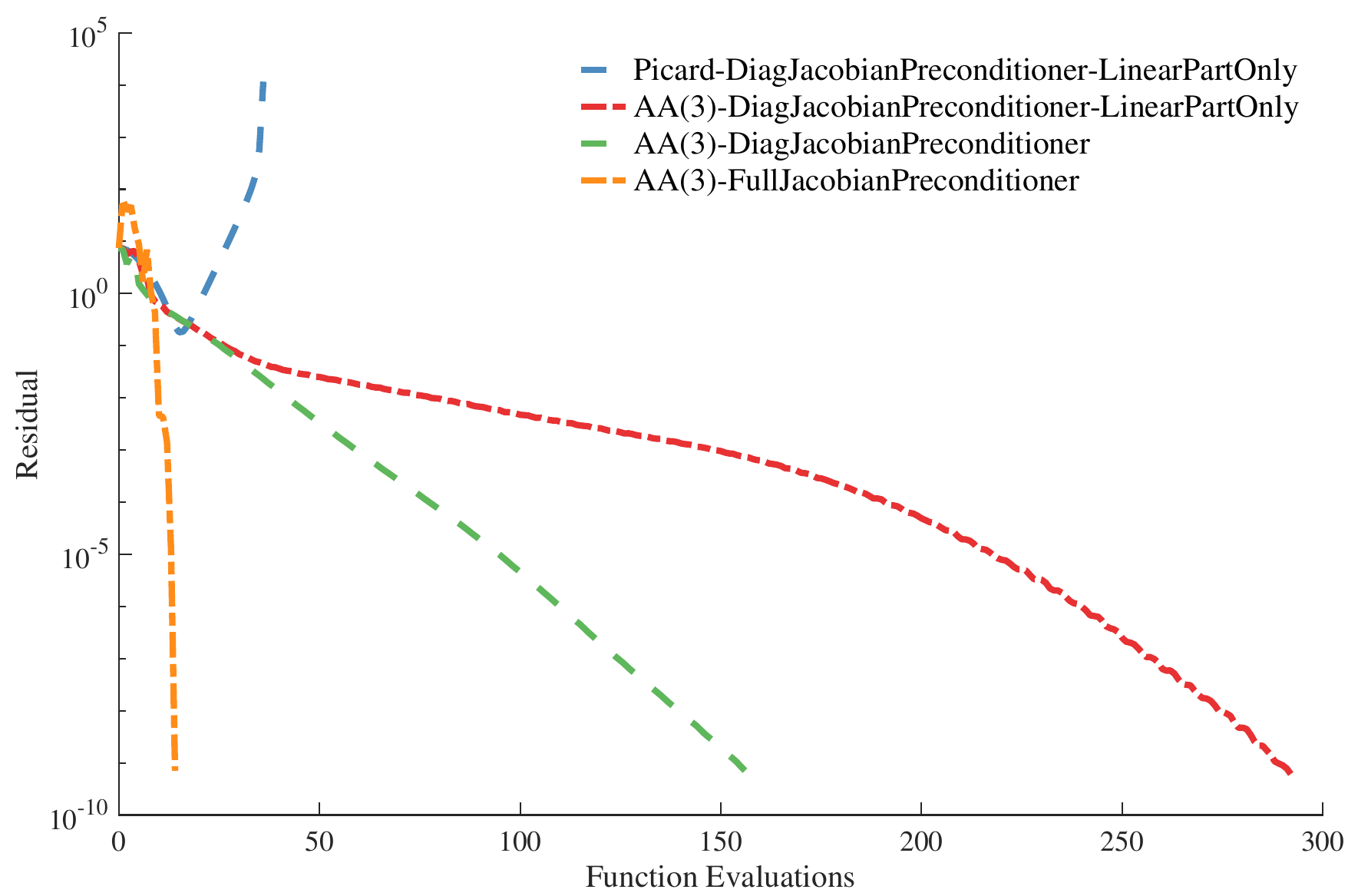}
    \label{fig:exampe4_50}}
    \caption{Using preconditioned AA to solve the stationary nonlinear convection-diffusion problem.}\label{diffusion}
\end{figure}

\section{Conclusions and outlook}
\label{sec:conclusions}
Our current study introduces a preconditioned Anderson acceleration scheme. This method provides the flexibility to adjust the balance between convergence rate (linear, super-linear, or quadratic) and the cost of computing or approximating the Jacobian matrix. It can recover different fixed-point iteration techniques such as the Picard, Newton, and quasi-Newton iterations. The preconditioned Anderson acceleration method is versatile and can be adapted to a wide range of problems with varying levels of nonlinearity. By incorporating an appropriate preconditioner, improved convergence rates and enhanced robustness can be achieved. We test several preconditioning strategies, including a constant preconditioner, diagonal Jacobian or block-diagonal Jacobian preconditioner, and full Jacobian preconditioner. Additionally, we discuss a delayed update strategy of the preconditioner, where it is updated every few steps to reduce the computational cost per iteration. 


Overall, our proposed framework is an effective tool for solving nonlinear problems, enabling users to tailor the convergence behavior to their specific needs while maintaining low computational costs.








\bibliography{wileyNJD-AMA}

\end{document}